\documentclass[12pt]{article}
\usepackage{mathrsfs}
\usepackage{amsfonts,amsmath,amsthm, amssymb}
\usepackage{latexsym, euscript, epic, eepic}
\usepackage{lineno}

\newtheorem{cro}{Corollary}[section]
\newtheorem{defn}{Definition}[section]

\newtheorem{thm}{Theorem}[section]
\newtheorem{lem}{Lemma}[section]

\begin{document}

\title{Conditional Variational Principle for Historic Set in Some
Nonuniformly Hyperbolic Systems
 \footnotetext {* Corresponding author}
  }
\author{Zheng Yin$^{1}$, Ercai Chen$^{*1,2}$ \\
  \small   1 School of Mathematical Sciences and Institute of Mathematics, Nanjing Normal University,\\
   \small   Nanjing 210023, Jiangsu, P.R.China\\
    \small 2 Center of Nonlinear Science, Nanjing University,\\
     \small   Nanjing 210093, Jiangsu, P.R.China.\\
      \small    e-mail: zhengyinmail@126.com ecchen@njnu.edu.cn
}
\date{}
\maketitle

\begin{center}
 \begin{minipage}{120mm}
{\small {\bf Abstract.} This article is devoted to the study
of the historic set, which was introduced by Ruelle,
of Birkhoff averges in some nonuniformly
hyperbolic systems via Pesin theory.
Particularly, we give a conditional variational principle
for historic sets. Our results can be applied
(i) to the diffeomorphisms on surfaces,
(ii) to the nonuniformly hyperbolic diffeomorphisms described by Katok
and several other classes of diffeomorphisms derived from Anosov systems. }
\end{minipage}
 \end{center}

\vskip0.5cm
\section{Introduction}
$(M,d,f)$ (or $(M,f)$ for short) is a topological dynamical system means that $(M,d)$ is a compact metric space together with a continuous self-map $f:M\to M.$
For a continuous function $\varphi:M\to\mathbb R$, $M$ can be divided into the following two parts:
\begin{align*}
M=\bigcup\limits_{\alpha\in\mathbb R}M(\varphi,\alpha)\cup\widehat{M}(\varphi,f),
\end{align*}
where
$$M(\varphi,\alpha)=\bigg\{x\in M:\lim\limits_{n\to\infty}\frac{1}{n}\sum\limits_{i=0}^{n-1}\varphi(f^ix)=\alpha\bigg\}$$
and
$$\widehat{M}(\varphi,f)=\bigg\{x\in M:\lim\limits_{n\to\infty}\frac{1}{n}\sum\limits_{i=0}^{n-1}\varphi(f^ix)\text{ does not exist}\bigg\}.$$
The level set $M(\varphi,\alpha)$ is so-called multifractal decomposition sets of  ergodic averages of $\varphi$ in multifractal analysis. There are fruitful results about the description of the structure ( Hausdorff dimension or topological entropy or  topological pressure) of these level sets. See e.g. \cite{BarSau,BarSauSch,Cli,FenHua,PeiChe,PfiSul,TakVer,Tho,ZhoChe}
and the references therein.

The set $\widehat{M}(\varphi,f)$ is called the historic set of  ergodic average of $\varphi.$ This terminology was introduced by Ruelle in \cite{Rue}. It is also called  non-typical points (see \cite{BarSch}), irregular set (see \cite{Tho1,Tho2}) and divergence  points (see \cite{CheKupShu}). By Birkhoff's ergodic theorem, $\widehat{M}(\varphi,f)$ is not detectable from the point of view of an invariant measure, i.e., for any invariant measure $\mu,$
\begin{align*}
\mu(\widehat{M}(\varphi,f))=0.
\end{align*}
However, Chen, Kupper and Shu \cite{CheKupShu} proved that $\widehat{X}(\varphi,f)$
is either empty or carries full entropy for maps with the specification property.  Thompson \cite{Tho1}
extended it to topological pressure for maps with the specification property.
In \cite{Tho2}, Thompson obtained the same result for maps with $g$-almost
product property, which can be applied to every $\beta$-shift.
This implies that $\widehat{M}(\varphi,f)$ is ``thick" in view of
topological entropy and topological pressure. Recently, Bomfim and Varandas \cite{BomVar} studied the upper bound estimates for topological
pressure of historic sets for weak Gibbs measures.
Motivated by their work, the aim of this paper is to study the topological entropy
of historic set in some nonuniformly hyperbolic systems via Pesin theory. Particularly, a conditional variational
principle is obtained for historic sets.

This article is organized as follows. In section 2, we provide some notions
and results of Pesin theory and state the main result. Section 3 is devoted
to the proof of the main results. Examples and applications are given in section 4.


\section{ Preliminaries}
In this section, we first present some notations to be used in this paper.
Then we introduce some notions and results of
Pesin theory \cite{BarPes2,KatHas,Pol} and state the main results.

We denote by $\mathscr M_{\rm inv}(M,f)$ and $\mathscr M_{\rm erg}(M,f)$
the set of all $f$-invariant Borel probability measures and ergodic measures
respectively. For an $f$-invariant subset $Z\subset X,$ let
$\mathscr M_{\rm inv}(Z,f)$ denote the subset of $\mathscr M_{\rm inv}(M,f)$
for which the measures $\mu$ satisfy $\mu(Z)=1$ and $\mathscr M_{\rm erg}(Z,f)$
denote those which are ergodic. Denote by $C^0(M)$ the space of continuous
functions from $M$ to $\mathbb{R}$ with the sup norm. For $\varphi\in C^0(M)$
and $n\geq1$ we denote $\sum_{i=0}^{n-1}\varphi(f^ix)$ by $S_n\varphi(x)$.
For every $\epsilon>0$, $n\in \mathbb{N}$ and a point $x\in M$, define
$B_n(x,\epsilon)=\{y\in M:d(f^ix,f^iy)<\epsilon,\forall 0\leq i\leq n-1\}$.
The $n$-ordered empirical measure of $x$ is
given by
\begin{align*}
\mathscr{E}_n(x)=\frac{1}{n}\sum\limits_{i=0}^{n-1}\delta_{f^ix},
\end{align*}
where $\delta_y$ is the Dirac mass at $y$. Denote by $V(x)$ the set
of limit measures of the sequence of measures $\mathscr{E}_n(x).$

Suppose $M$ is a compact connected boundary-less Riemannian
$n$-dimension manifold and $f:X\to X$ is a $C^{1+\alpha}$
diffeomorphism. Let $\mu\in\mathscr M_{\rm erg}(Z,f)$ and $Df_x$ denote
the tangent map of $f$ at $x\in M.$ We say that
$x\in X$ is a regular point of $f$ if there exist
$\lambda_1(\mu)>\lambda_2(\mu)>\cdots>\lambda_{\phi(\mu)}(\mu)$ and a
decomposition on the tangent space $T_x
M=E_1(x)\oplus\cdots\oplus E_{\phi(\mu)}(x)$ such that
\begin{align*}
\lim\limits_{n\to\infty}\frac{1}{n}\log\|(Df^n_x)u\|=\lambda_j(x),
\end{align*}
where $0\neq u\in E_j(x), 1\leq j\leq \phi(\mu).$ The number $\lambda_j(x)$ and
the space $E_j(x)$ are called the Lyapunov exponents and the eigenspaces of
$f$ at the regular point $x,$ respectively. Oseledets theorem \cite{Ose} say that all
regular points forms a Borel set with total measure. For a regular point
$x\in M$, we define
\begin{align*}
\lambda^+(\mu)=\min\{\lambda_i(\mu)|\lambda_i(\mu)\geq0,1\leq i\leq \phi(\mu)\}
\end{align*}
and
\begin{align*}
\lambda^-(\mu)=\min\{-\lambda_i(\mu)|\lambda_i(\mu)\leq0,1\leq i\leq \phi(\mu)\}.
\end{align*}
We appoint $\min\emptyset=0$. An ergodic measure $\mu$ is hyperbolic if $\lambda^+(\mu)$
and $\lambda^-(\mu)$ are both non-zero.

\begin{defn}
Given $\beta_1,\beta_2\gg\epsilon>0$ and for all $k\in\mathbb{Z}^+,$
the hyperbolic block $\Lambda_k=\Lambda_k(\beta_1,\beta_2,\epsilon)$
consists of all points $x\in M$  such that there exists a
decomposition $T_xM=E_x^s\oplus E_x^u$ satisfying:
\begin{itemize}
  \item $Df^t(E_x^s)=E^s_{f^tx}$ and $Df^t(E_x^u)=E^u_{f^tx};$
  \item $\|Df^n|E^s_{f^tx}\|\leq e^{\epsilon k}e^{-(\beta_1-\epsilon)n}e^{\epsilon|t|},\forall t\in\mathbb{Z},n\geq1;$
  \item $\|Df^{-n}|E^u_{f^tx}\|\leq e^{\epsilon k}e^{-(\beta_2-\epsilon)n}e^{\epsilon|t|},\forall t\in\mathbb{Z},n\geq1;$
  \item $\tan (\angle(E^s_{f^tx},E^u_{f^tx}))\geq e^{-\epsilon k}e^{-\epsilon|t|},\forall t\in\mathbb{Z}.$
\end{itemize}
\end{defn}

\begin{defn}
$
\Lambda(\beta_1,\beta_2,\epsilon)=\bigcup\limits_{k=1}^\infty\Lambda_k(\beta_1,\beta_2,\epsilon)
$ is a Pesin set.
\end{defn}
The following statements are elementary properties of Pesin blocks (see \cite{Pol}):
\begin{itemize}
  \item[(1)] $\Lambda_1\subseteq\Lambda_2\subseteq\cdots;$
  \item[(2)] $f(\Lambda_k)\subseteq\Lambda_{k+1},f^{-1}(\Lambda_k)\subseteq\Lambda_{k+1};$
  \item[(3)] $\Lambda_k$ is compact for each $k\geq1$;
  \item[(4)] For each $k\geq1$, the splitting $\Lambda_k\ni x\mapsto E_x^s\oplus E_x^u$ is continuous.
\end{itemize}
The Pesin set $\Lambda(\beta_1,\beta_2,\epsilon)$ is an
$f$-invariant set but usually not compact. Given an ergodic measure
$\mu\in\mathscr M_{\rm erg}(M,f)$, denote by $\mu|\Lambda_l$ the
conditional measure of $\mu$ on $\Lambda_l.$ Let
$\widetilde{\Lambda}_l=$ supp$(\mu|\Lambda_l)$ and
$\widetilde{\Lambda}_\mu=\bigcup_{l\geq1}\widetilde{\Lambda}_l.$
If $\omega$ is an ergodic hyperbolic measure for $f$ and $\beta_1\leq\lambda^-(\omega)$ and
$\beta_2\leq\lambda^+(\omega)$, then $\omega\in \mathscr M_{\rm inv}(\widetilde{\Lambda}_\omega,f)$.

Let $\{\delta_k\}_{k=1}^\infty $ be a sequence of positive real
numbers. Let $\{x_n\}_{n=-\infty}^\infty$ be a sequence of points in
$\Lambda=\Lambda(\beta_1,\beta_2,\epsilon)$ for which there exists a
sequence $\{s_n\}_{n=-\infty}^{\infty}$ of positive integers satisfying:
\begin{equation*}\begin{split}
&\text{(a) } x_n\in\Lambda_{s_n},\forall n\in\mathbb{Z};\\
&\text{(b) } |s_n-s_{n-1}|\leq 1, \forall n\in\mathbb{Z};\\
&\text{(c) } d(f(x_n),x_{n+1})\leq\delta_{s_n},  \forall n\in\mathbb{Z},
\end{split}\end{equation*}
then we call $\{x_n\}_{n=-\infty}^\infty$ a $\{\delta_k\}_{k=1}^\infty$ pseudo-orbit.  Given $\eta>0$
a point $x\in M$ is an $\eta$-shadowing point for the
$\{\delta_k\}_{k=1}^\infty$ pseudo-orbit  if $d(f^n(x),x_n)\leq
\eta\epsilon_{s_n},\forall n\in\mathbb{Z},$ where
$\epsilon_k=\epsilon_0e^{-\epsilon k}$ and $\epsilon_0$ is a constant only dependent on
the system of $f$.

\vskip0.3cm

\noindent \textbf{Weak shadowing lemma.} \cite{Hir,KatHas,Pol}
{\it
Let $f:M\to M$ be a $C^{1+\alpha}$ diffeomorphism, with a non-empty
Pesin set $\Lambda=\Lambda(\beta_1,\beta_2,\epsilon)$ and fixed
parameters, $\beta_1,\beta_2\gg \epsilon>0.$ For $\eta>0$ there exists
a sequence $\{\delta_k\}$ such that for any $\{\delta_k\}$
pseudo-orbit there exists a unique $\eta$-shadowing point.
}
\vskip0.3cm

\textbf{Bowen's topological entropy}
Bowen introduced his concept of topological entropy in \cite{Bow}.
This study defines it in an alternative way for convenience
\cite{Pes}. Given
$Z\subset M,\epsilon>0$ and $N\in \mathbb{N},$ let
$\Gamma_n(Z,\epsilon)$ be the collection of all finite or countable
covers of $Z$ by sets of the form $B_n(x,\epsilon)$ with $n\geq N$.
For each $s\in\mathbb{R},$ we set
\begin{align*}
m(Z,s,N,\epsilon)=\inf\bigg\{\sum\limits_{B_n(x,\epsilon)\in\mathcal{C}}e^{-ns}:\mathcal{C}\in \Gamma_n(Z,\epsilon)\bigg\},
\end{align*}
and
\begin{align*}
m(Z,s,\epsilon)=\lim_{N\to\infty}m(Z,s,N,\epsilon).
\end{align*}
Define
\begin{align*}
h_{top}(Z,\epsilon)&=\inf\{s\in\mathbb{R}:m(Z,s,\epsilon)=0\}=\sup\{s\in\mathbb{R}:m(Z,s,\epsilon)=\infty\},
\end{align*}
and topological entropy of $Z$ is
\begin{align*}
h_{top}(Z):=\lim_{\epsilon\rightarrow0}h_{top}(Z,\epsilon).
\end{align*}
Now, we state the main result of this paper as follows:
\begin{thm}\label{thm2.1}
Let $f:M\to M$ be a $C^{1+\alpha}$ diffeomorphism of a compact Riemannian manifold, with a non-empty
Pesin set $\Lambda=\Lambda(\beta_1,\beta_2,\epsilon)$ and fixed
parameters, $\beta_1,\beta_2\gg \epsilon>0$ and
let $\mu\in\mathscr M_{\rm erg}(M,f)$ be
any ergodic measure. Let
\begin{align}\label{equ2.1}
N(\widetilde{\Lambda}_\mu)=\{x\in M:V(x)\cap \mathscr{M}_{inv}(\widetilde{\Lambda}_\mu,f)\neq\emptyset\}.
\end{align}
For $\varphi\in C^0(M)$, one of the following conclusions is right.
\begin{enumerate}
\item The function $\nu\mapsto\int\varphi d\nu$ is constant for $\nu\in \mathscr M_{\rm inv}(\widetilde{\Lambda}_\mu,f)$.
\item $\widehat{M}(\varphi|N(\widetilde{\Lambda}_\mu),f)\neq\emptyset$ and
\begin{align*}
h_{top}\left(\widehat{M}(\varphi|N(\widetilde{\Lambda}_\mu),f)\right)=\sup\left\{h_\nu(f):\nu\in \mathscr{M}_{inv}(\widetilde{\Lambda}_\mu,f)\right\},
\end{align*}
\end{enumerate}
where $\widehat{M}(\varphi|N(\widetilde{\Lambda}_\mu),f)=\widehat{M}(\varphi,f)\cap N(\widetilde{\Lambda}_\mu)$.
\end{thm}
The condition (\ref{equ2.1}) is motivated by the work of Pesin \& Pitskel \cite{PesPit} and we do not require the measures are ergodic.
From theorem \ref{thm2.1}, we obtain that if the function $\nu\mapsto\int\varphi d\nu$ is not constant for 
$\nu\in \mathscr M_{\rm inv}(\widetilde{\Lambda}_\mu,f)$, we have $\widehat{M}(\varphi,f)\neq\emptyset$ and 
\begin{align*}
h_{top}(\widehat{M}(\varphi,f))\geq\sup\left\{h_\nu(f):\nu\in \mathscr{M}_{inv}(\widetilde{\Lambda}_\mu,f)\right\}.
\end{align*}

\begin{cro}\label{cro2.1}
Let $f:M\to M$ be a $C^{1+\alpha}$ diffeomorphism of a compact Riemannian manifold
and let $\omega\in\mathscr M_{\rm erg}(M,f)$ be a hyperbolic measure. For
$\beta_1\leq\lambda^-(\omega)$ and $\beta_2\leq\lambda^+(\omega)$,
let $\widetilde{\Lambda}_\omega=\bigcup_{l\geq1}{\rm supp}(\omega|\Lambda_l(\beta_1,\beta_2,\epsilon))$.
If $\varphi\in C^0(M)$, one of the following conclusions is right.
\begin{enumerate}
\item The function $\nu\mapsto\int\varphi d\nu$ is constant for $\nu\in \mathscr M_{\rm inv}(\widetilde{\Lambda}_\omega,f)$.
\item $\widehat{M}(\varphi|N(\widetilde{\Lambda}_\omega),f)\neq\emptyset$
and 
\begin{align*}
h_{top}(\widehat{M}(\varphi|N(\widetilde{\Lambda}_\omega),f))=\sup\left\{h_\nu(f):\nu\in \mathscr{M}_{inv}(\widetilde{\Lambda}_\omega,f)\right\}.
\end{align*}
\end{enumerate}
\end{cro}

\section{Proof of Main Result}
In this section, we will verify theorem \ref{thm2.1}.
To obtain the lower bound estimate we need to construct
a suitable pseudo-orbit and a sequence of measures
to apply entropy distribution principle.
Our method is inspired by \cite{LiaLiaSUnTia},
\cite{PfiSul} and \cite{Tho2}.

\subsection{Katok's Definition of Metric Entropy}
We use the Katok's definition of Metric Entropy based on the following
lemma.

\begin{lem}\label{lem3.1}{\rm \cite{Kat2}}
Let $(M,d)$ be a compact metric space, $f:M\to M$ be a continuous map and $\nu$
be an ergodic invariant measure. For $\epsilon>0$, $\delta\in (0,1)$ let $N^\nu(n,\epsilon,\delta)$ denote the minimum
number of $\epsilon$-Bowen balls $B_n(x,\epsilon)$, which cover a set of $\nu$-measure at least $1-\delta$. Then
\begin{align*}
h_\nu(f)=\lim_{\epsilon\to0}\liminf_{n\to\infty}\frac{1}{n}\log N^\nu(n,\epsilon,\delta)=
\lim_{\epsilon\to0}\limsup_{n\to\infty}\frac{1}{n}\log N^\nu(n,\epsilon,\delta).
\end{align*}
\end{lem}

Fix $\delta\in (0,1)$. For $\epsilon>0$ and $\nu\in\mathscr{M}_{erg}(M,f)$, we define
\begin{align*}
h_\nu^{Kat}(f,\epsilon):=\liminf_{n\to\infty}\frac{1}{n}\log N^\nu(n,\epsilon,\delta).
\end{align*}
Then by lemma \ref{lem3.1},
\begin{align*}
h_\nu(f)=\lim_{\epsilon\to0}h_\nu^{Kat}(f,\epsilon).
\end{align*}
If $\nu$ is non-ergodic, we will define $h_\nu^{Kat}(f,\epsilon)$ by the ergodic
decomposition of $\nu$. The following lemma is necessary.

\begin{lem}\label{lem3.2}
Fix $\epsilon,\delta>0$ and $n\in\mathbb{N}$, the function $s:\mathscr{M}_{erg}(M,f)\to\mathbb{R}$
defined by $\nu\mapsto N^\nu(n,\epsilon,\delta)$ is upper semi-continuous.
\end{lem}
\begin{proof}
Let $\nu_k\to\nu$. Let $a>N^\nu(n,\epsilon,\delta)$; then there exists a set $S$
which $(n,\epsilon)$ span some set $Z$ with $\nu(Z)>1-\delta$ such that
$a>\#S$, where $\#S$ denote the number of elements in $S$.
If $k$ is large enough, then $\nu_k(\bigcup_{x\in S}B_n(x,\epsilon))>1-\delta$, which implies that
\begin{align*}
a>N^{\nu_{k}}(n,\epsilon,\delta).
\end{align*}
Thus we obtain
\begin{align*}
N^\nu(n,\epsilon,\delta)\geq\limsup_{k\to\infty}N^{\nu_{k}}(n,\epsilon,\delta),
\end{align*}
which completes the proof.
\end{proof}

Lemma \ref{lem3.2} tells us that the function $\overline{s}:\mathscr{M}_{erg}(M,f)\to\mathbb{R}$
defined by
\begin{align*}
\overline{s}(m)=h_m^{Kat}(f,\epsilon)
\end{align*}
is measurable.
Assume $\nu=\int_{\mathscr{M}_{erg}(M,f)}md\tau(m)$ is the ergodic decomposition of $\nu$.
Define
\begin{align*}
h_\nu^{Kat}(f,\epsilon):=\int_{\mathscr{M}_{erg}(M,f)}h_m^{Kat}(f,\epsilon)d\tau(m).
\end{align*}
By dominated convergence theorem, we have
\begin{align}\label{equ3.1}
h_\nu(f)=\int_{\mathscr{M}_{erg}(M,f)}\lim_{\epsilon\to0}h_m^{Kat}(f,\epsilon)d\tau(m)
=\lim_{\epsilon\to0}h_\nu^{Kat}(f,\epsilon).
\end{align}


\subsection{Some Lemmas}
For $\mu,\nu\in \mathscr{M}(M),$ define a compatible metric $D$ on
$\mathscr{M}(M)$ as follows:
\begin{align*}
D(\mu,\nu):=\sum\limits_{i\geq1}\frac{|\int\varphi_id\mu-\int\varphi_id\nu|}{2^{i+1}\|\varphi_i\|}
\end{align*}
where $\{\varphi_i\}_{i=1}^\infty$ is the dense subset of
$C^0(M)$. It is obvious that $D(\mu,\nu)\leq1$ for any $\mu,\nu\in \mathscr{M}(M).$

\begin{lem}\label{lem3.3}
Fix $\epsilon>0$.
For any integer $k\geq1$ and invariant measure
$\nu\in\mathscr{M}_{inv}(\widetilde{\Lambda}_{\mu},f)$,
there exists a finite convex combination of ergodic probability measures with
rational coefficients $\mu_k=\sum\limits_{j=1}^{s_k}a_{k,j}m_{k,j}$ such that
\begin{align*}
D(\nu,\mu_k)\leq\frac{1}{k}, m_{k,j}(\widetilde{\Lambda}_{\mu})=1, ~{\rm and} ~
h_\nu^{Kat}(f,\epsilon)\leq\sum_{j=1}^{s_{k}}a_{k,j}h_{m_{k,j}}^{Kat}(f,\epsilon).
\end{align*}
\end{lem}
\begin{proof}
Let
\begin{align*}
\nu=\int_{\mathscr{M}_{erg}(\widetilde{\Lambda}_{\mu},f)}md\tau(m)
\end{align*}
be the ergodic decomposition of $\nu$. Choose $N$ large enough such that
\begin{align*}
\sum\limits_{n=N+1}^{\infty}\frac{2}{2^{n+1}}<\frac{1}{3k}.
\end{align*}
We choose $\zeta>0$ such that $D(\nu_1,\nu_2)<\zeta$ implies that
\begin{align*}
\left|\int \varphi_n d\nu_1-\int \varphi_n d\nu_2\right|<\frac{\|\varphi_n\|}{3k},n=1,2,\cdots,N.
\end{align*}
Let $\{A_{k,1},A_{k,2},\cdots,A_{k,s_k}\}$ be a partition of $\mathscr{M}_{erg}(\widetilde{\Lambda}_{\mu},f)$
with diameter smaller than $\zeta$. For any $A_{k,j}$ there exists an ergodic $m_{k,j}\in A_{k,j}$ such that
\begin{align*}
\int_{A_{k,j}}h_m^{Kat}(f,\epsilon)d\tau(m)
\leq\tau(A_{k,j})h_{m_{k,j}}^{Kat}(f,\epsilon).
\end{align*}
Obviously $m_{k,j}(\widetilde{\Lambda}_{\mu})=1$ and
$h_{\nu}^{Kat}(f,\epsilon)\leq\sum_{j=1}^{s_{k}}\tau(A_{k,j})h_{m_{k,j}}^{Kat}(f,\epsilon)$.
Let us choose rational numbers $a_{k,j}>0$
such that
$$|a_{k,j}-\tau(A_{k,j})|<\frac{1}{3ks_k}$$
and
$$h_{\nu}^{Kat}(f,\epsilon)\leq\sum_{j=1}^{s_{k}}a_{k,j}h_{m_{k,j}}^{Kat}(f,\epsilon).$$
Let
\begin{align*}
\mu_k=\sum\limits_{j=1}^{s_k}a_{k,j}m_{k,j}.
\end{align*}
By ergodic decomposition theorem, one can readily verify that
\begin{align*}
\left|\int\varphi_n d\nu-\int \varphi_n d\mu_k\right|\leq\frac{2\|\varphi_n\|}{3k},n=1,\cdots,N.
\end{align*}
Thus, we obtain
\begin{align*}
D(\nu,\mu_k)\leq\frac{1}{k}.
\end{align*}
\end{proof}

\begin{lem}{\rm \cite{Boc}}\label{lem3.4}
Let $f:M\to M$ be a $C^{1}$ diffeomorphism of a compact Riemannian manifold and $\mu\in\mathscr M_{\rm inv}(M,f)$.
Let $\Gamma\subseteq M$ be a measurable set with $\mu(\Gamma)>0$ and let
\begin{align*}
\Omega=\bigcup_{n\in \mathbb{Z}}f^n(\Gamma).
\end{align*}
Take $\gamma>0$. Then there exists a measurable function $N_0:\Omega\to \mathbb{N}$ such that for a.e.$x\in\Omega$
and every $t\in[0,1]$ there is some $l\in\{0,1,\cdots,n\}$ such that $f^l(x)\in\Gamma$ and
$\left|(l/n)-t\right|<\gamma$.
\end{lem}

\begin{lem}(Entropy Distribution Principle {\rm \cite{Tho2}})
Let $f:M\to M$ be a continuous transformation. Let $Z\subseteq M$ be an arbitrary Borel set. Suppose there exists $\epsilon>0$
and $s\geq0$ such that one can find a sequence of Borel probability measures $\mu_k$, a constant $K>0$ and an integer $N$ satisfying
\begin{align*}
\limsup_{k\to\infty}\mu_{k}(B_n(x,\epsilon))\leq Ke^{-ns}
\end{align*}
for every ball $B_n(x,\epsilon)$ such that $B_n(x,\epsilon)\cap Z\neq\emptyset$ and $n\geq N$. Furthermore, assume that at least one  limit measure $\nu$ of the sequence $\mu_k$ satisfies $\nu(Z)>0$. Then $h_{top}(Z,\epsilon)>s$.
\end{lem}

\subsection{Proof of Theorem \ref{thm2.1}}
It suffices to consider the case that the function $\nu\mapsto\int\varphi d\nu$ is
not constant for $\nu\in \mathscr M_{\rm inv}(\widetilde{\Lambda}_\mu,f)$.
Fix small $0<\gamma<1$ and $0<\delta<1$.  Let $\mathbf{C}:=\sup\{h_\nu(f):\nu\in \mathscr{M}_{inv}(\widetilde{\Lambda}_\mu,f)\}$. Obviously,
$\mathbf{C}$ is finite.  Choose  $\mu_1\in\mathscr{M}_{inv}(\widetilde{\Lambda}_\mu,f)$
such that
$$h_{\mu_1}(f)>\mathbf{C}-\gamma/3$$
and $\mu'\in \mathscr{M}_{inv}(\widetilde{\Lambda}_\mu,f)$ satisfies
$\int\varphi d\mu_1\neq \int\varphi d\mu'$. Let $\mu_2=t_1\mu_1+t_2\mu'$
where $t_1+t_2=1$ and $t_1\in(0,1)$ is chosen sufficiently close to 1 so that
$$h_{\mu_2}(f)>\mathbf{C}-2\gamma/3.$$
Obviously, $\int\varphi d\mu_1\neq \int\varphi d\mu_2$.
By (\ref{equ3.1}), we can choose $\epsilon'>0$ sufficiently small so
\begin{align*}
h^{Kat}_{\mu_1}(f,\epsilon')>\mathbf{C}-\gamma \text{ and }
h^{Kat}_{\mu_2}(f,\epsilon')>\mathbf{C}-\gamma.
\end{align*}
Let $\rho:\mathbb N \to\{1,2\}$ be given by $\rho(k)=(k+1)(\text{mod }2)+1$.
The following lemma can be easily obtained from lemma \ref{lem3.3}.

\begin{lem}\label{lem3.6}
For any integer $k\geq1$ and $\mu_1,\mu_2\in\mathscr{M}_{inv}(\widetilde{\Lambda}_\mu,f)$,
there exists a finite convex combination of ergodic probability measures with
rational coefficients $\nu_k=\sum\limits_{j=1}^{s_k}a_{k,j}m_{k,j}$ such that
\begin{align*}
D(\mu_{\rho(k)},\nu_k)\leq\frac{1}{k}, m_{k,j}(\widetilde{\Lambda}_{\mu})=1, ~{\rm and} ~
h_{\mu_{\rho(k)}}^{Kat}(f,\epsilon')\leq\sum_{j=1}^{s_{k}}a_{k,j}h_{m_{k,j}}^{Kat}(f,\epsilon').
\end{align*}
\end{lem}

We choose a increasing sequence $l_k\to\infty$ such that
$m_{k,j}(\widetilde{\Lambda}_{l_k})>1-\gamma$ for all $1\leq j\leq s_k$.
Let $\eta=\frac{\epsilon'}{4\epsilon_0},$
it follows from weak shadowing lemma that there is a sequence of numbers $\{\delta_k\}.$
Let $\xi_k$ be a finite partition of $X$ with diam$(\xi_k)<\frac{\delta_{l_k}}{3}$
and $\xi_k\geq\{\widetilde{\Lambda}_{l_k}, M\setminus \widetilde{\Lambda}_{l_k} \}.$
For $n\in\mathbb{N},$ we consider the set
\begin{align*}
\Lambda^n(m_{k,j})=\{x\in\widetilde{\Lambda}_{l_k}:f^q(x)\in\xi_k(x){\rm ~for~some~} q\in[n,(1+\gamma)n] \\
{\rm~and~}D(\mathscr{E}_m(x),m_{k,j})<\frac{1}{k} {\rm~for~all~} m\geq n\},
\end{align*}
where $\xi_k(x)$ is the element in $\xi_k$ containing $x.$
By Birkhoff ergodic theorem and lemma \ref{lem3.4} we have
$m_{k,j}(\Lambda^n(m_{k,j}))\to m_{k,j}(\widetilde{\Lambda}_{l_k})$ as $n\to\infty$.
So, we can take $n_k\to\infty$ such that
\begin{align*}
m_{k,j}(\Lambda^n(m_{k,j}))>1-\delta
\end{align*}
for all $n\geq n_k$ and $1\leq j\leq s_k.$

For $k\in \mathbb{N}$, let
\begin{align*}
Q(\Lambda^n(m_{k,j}),\epsilon')&=\inf\{\sharp S:S \text{ is } (n,\epsilon') \text{ spanning set for } \Lambda^n(m_{k,j}) \},\\
P(\Lambda^n(m_{k,j}),\epsilon')&=\sup\{\sharp S:S \text{ is } (n,\epsilon') \text{ separated set for } \Lambda^n(m_{k,j}) \}.
\end{align*}
Then for all $n\geq n_k$ and $1\leq j\leq s_k$, we have
\begin{align*}
P(\Lambda^n(m_{k,j}),\epsilon')\geq Q(\Lambda^n(m_{k,j}),\epsilon')\geq N^{m_{k,j}}(n,\epsilon',\delta).
\end{align*}
We obtain
\begin{align*}
\liminf_{n\to\infty}\frac{1}{n}\log P(\Lambda^n(m_{k,j}),\epsilon')
\geq h^{Kat}_{m_{k,j}}(f,\epsilon').
\end{align*}
Thus for each $k\in \mathbb{N}$, we can choose $t_k$ large enough such that $\exp(\gamma t_k)>\sharp \xi_k$
and
\begin{align*}
\frac{1}{t_k}\log P(\Lambda^{t_k}(m_{k,j}),\epsilon')>h^{Kat}_{m_{k,j}}(f,\epsilon')-\gamma
\end{align*}
for $1\leq j\leq s_k$. Let $S(k,j)$
be a $(t_k,\epsilon')$-separated set for $\Lambda^{t_k}(m_{k,j})$ and
\begin{align*}
\# S(k,j)\geq\exp\left(t_k(h^{Kat}_{m_{k,j}}(f,\epsilon')-2\gamma)\right).
\end{align*}
For each $q\in[t_k,(1+\gamma)t_k],$ let
\begin{align*}
V_q=\{x\in S(k,j):f^q(x)\in\xi_k(x)\}
\end{align*}
and let $n=n(k,j)$ be the value of $q$ which maximizes $\#V_q.$ Obviously,
$n\geq t_k$ and $t_k\geq\frac{n}{1+\gamma}\geq n(1-\gamma).$
Since $\exp({\gamma t_k})\geq\gamma t_k+1,$ we have that
\begin{align*}
\#V_n\geq\frac{\#S(k,j)}{\gamma t_k+1}\geq \exp\left(t_k(h^{Kat}_{m_{k,j}}(f,\epsilon')-3\gamma)\right).
\end{align*}
Consider the element $A_n(m_{k,j})\in\xi_k$ such that $\#(V_n\cap A_n(m_{k,j}))$ is maximal.
Let $W_{n(k,j)}=V_{n(k,j)}\cap A_{n(k,j)}(m_{k,j})$. It follows that
\begin{align*}
\#W_{n(k,j)}\geq \frac{1}{\#\xi_k}\#V_n\geq \frac{1}{\#\xi_k}\exp\left(t_k(h^{Kat}_{m_{k,j}}(f,\epsilon')-3\gamma)\right).
\end{align*}
Since $\exp(\gamma t_k)>\sharp \xi_k$, $t_k\geq n(k,j)(1-\gamma)$ and $\#W_{n(k,j)}\geq1$, we have
\begin{align*}
\#W_{n(k,j)}\geq\exp\left(n(k,j)(1-\gamma)(h^{Kat}_{m_{k,j}}(f,\epsilon')-4\gamma)\right).
\end{align*}
Notice that $A_{n(k,j)}(m_{k,j})$ is contained in an open set $U(k,j)$ with  diam$(U(k,j))\leq3\text{diam}(\xi_k).$
By the ergodicity of $\mu,$ for any two measures $m_{k_1,j_1},m_{k_2,j_2}$ and any natural number $N$,
there exists $s=s(k_1,j_1,k_2,j_2)>N$ and $y=y(k_1,j_1,k_2,j_2)\in U(k_1,j_1)\cap\widetilde{\Lambda}_{l_{k_1}}$  such that
$f^s(y)\in U(k_2,j_2)\cap\widetilde{\Lambda}_{l_{k_2}}.$
Letting $C_{k,j}=\frac{a_{k,j}}{n(k,j)},$ we can choose an integer $N_k$ large enough so that $N_kC_{k,j}$ are integers and
\begin{align}\label{equ1}
N_k\geq k\sum\limits_{\substack{1\leq r_1,r_2\leq k+1 \\ 1\leq j_i\leq s_{r_i},i=1,2}}s(r_1,j_1,r_2,j_2).
\end{align}
Let $X_k=\sum\limits_{j=1}^{s_k-1}s(k,j, k,j+1)+s(k,s_k, k,1)$  and
\begin{align}\label{equ2}
Y_k=\sum\limits_{j=1}^{s_k}N_kn(k,j)C_{k,j}+X_k=N_k+X_k,
\end{align}
then we have
\begin{align}
\frac{N_k}{Y_k}\geq \frac{1}{1+\frac{1}{k}}\geq 1-\frac{1}{k}.\label{equ3}
\end{align}
Choose a strictly increasing sequence $\{T_k\}$ with $T_k\in\mathbb{N},$
\begin{align}\label{equ4}
Y_{k+1}\leq\frac{1}{k+1}\sum\limits_{r=1}^kY_rT_r,
\sum\limits_{r=1}^k(Y_rT_r+s(r,1,r+1,1))\leq \frac{1}{k+1}Y_{k+1}T_{k+1}.
\end{align}
For $x\in X,$ we define segments of orbits
\begin{align*}
L_{k,j}(x)&:=(x,f(x),\cdots,f^{n(k,j)-1}(x)), 1\leq j\leq s_k,\\
\widehat{L}_{k_1,j_1,k_2,j_2}(x)&:=(x,f(x),\cdots,f^{s(k_1,j_1,k_2,j_2)-1}(x)),1\leq j_i\leq s_{k_i},i=1,2.
\end{align*}
Consider the pseudo-orbit with finite length
\begin{align*}
O_k=O( &x(1,1,1,1), \cdots,x(1,1,1,N_1C_{1,1}),\cdots,
x(1,s_1,1,1),\cdots,x(1,s_1,1,N_1C_{1,s_1});\\&\cdots;\\
&x(1,1,T_1,1),\cdots,x(1,1,T_1,N_1C_{1,1}),\cdots,x(1,s_1,T_1,1),\cdots, x(1,s_1,T_1,N_1C_{1,s_1});\\
&\vdots\\&
x(k,1,1,1), \cdots,x(k,1,1,N_kC_{k,1}),\cdots,
x(k,s_k,1,1),\cdots,x(k,s_k,1,N_kC_{k,s_k});\\&\cdots;\\&
x(k,1,T_k,1), \cdots,x(k,1,T_k,N_kC_{k,1}),\cdots,
x(k,s_k,T_k,1),\cdots,x(k,s_k,T_k,N_kC_{k,s_k});
)
\end{align*}
with the precise form as follows:
\begin{align*}
\{ &L_{1,1}(x(1,1,1,1)), \cdots,L_{1,1}(x(1,1,1,N_1C_{1,1})),\widehat{L}_{1,1,1,2}(y(1,1,1,2));\\&
L_{1,2}(x(1,2,1,1)), \cdots,L_{1,2}(x(1,2,1,N_1C_{1,2})),\widehat{L}_{1,2,1,3}(y(1,2,1,3));
\cdots,\\&
L_{1,s_1}(x(1,s_1,1,1)), \cdots,L_{1,s_1}(x(1,s_1,1,N_1C_{1,s_1})),\widehat{L}_{1,s_1,1,1}(y(1,s_1,1,1));\\&
\cdots,\\&
L_{1,1}(x(1,1,T_1,1)), \cdots,L_{1,1}(x(1,1,T_1,N_1C_{1,1})),\widehat{L}_{1,1,1,2}(y(1,1,1,2));\\&
L_{1,2}(x(1,2,T_1,1)), \cdots,L_{1,2}(x(1,2,T_1,N_1C_{1,2})),\widehat{L}_{1,2,1,3}(y(1,2,1,3));
\cdots,\\&
L_{1,s_1}(x(1,s_1,T_1,1)), \cdots,L_{1,s_1}(x(1,s_1,T_1,N_1C_{1,s_1})),\widehat{L}_{1,s_1,1,1}(y(1,s_1,1,1));\\&
\widehat{L}(y(1,1,2,1));\\&\vdots,\\&
L_{k,1}(x(k,1,1,1)), \cdots,L_{k,1}(x(k,1,1,N_kC_{k,1})),\widehat{L}_{k,1,k,2}(y(k,1,k,2));\\&
L_{k,2}(x(k,2,1,1)), \cdots,L_{k,2}(x(k,2,1,N_kC_{k,2})),\widehat{L}_{k,2,k,3}(y(k,2,k,3));\cdots\\&
L_{k,s_k}(x(k,s_k,1,1)), \cdots,L_{k,s_k}(x(k,s_k,1,N_kC_{k,s_k})),\widehat{L}_{k,s_k,k,1}(y(k,s_k,k,1));\\
&\cdots\\&
L_{k,1}(x(k,1,T_k,1)), \cdots,L_{k,1}(x(k,1,T_k,N_kC_{k,1})),\widehat{L}_{k,1,k,2}(y(k,1,k,2));\\&
L_{k,2}(x(k,2,T_k,1)), \cdots,L_{k,2}(x(k,2,T_k,N_kC_{k,2})),\widehat{L}_{k,2,k,3}(y(k,2,k,3));\cdots\\&
L_{k,s_k}(x(k,s_k,T_k,1)), \cdots,L_{k,s_k}(x(k,s_k,T_k,N_kC_{k,s_k})),\widehat{L}_{k,s_k,k,1}(y(k,s_k,k,1));\\&
\widehat{L}(y(k,1,k+1,1));
\},
\end{align*}
where $x(q,j,i,t)\in W_{n(q,j)}.$

For $1\leq q\leq k,1\leq i\leq T_q,1\leq j\leq s_q, 1\leq t\leq N_qC_{q,j},$ let $M_1=0,$
\begin{align*}
M_q&=M_{q,1}=\sum\limits_{r=1}^{q-1}(T_rY_r+s(r,1,r+1,1)),\\
M_{q,i}&=M_{q,i,1}=M_q+(i-1)Y_q,\\
M_{q,i,j}&=M_{q,i,j,1}=M_{q,i}+\sum\limits_{p=1}^{j-1}(N_qn(q,p)C_{q,p}+s(k,p,k,p+1)),
\\
M_{q,i,j,t}&=M_{q,i,j}+(t-1)n(q,j).
\end{align*}
By weak shadowing lemma, there exist at least one shadowing point $z$ of $O_k$ such that
\begin{align*}
d(f^{M_{q,i,j,t}+p}(z),f^p(x(q,j,i,t)))\leq \eta\epsilon_0\exp(-\epsilon l_q)\leq\frac{\epsilon'}{4\epsilon_0}\epsilon_0\exp(-\epsilon l_q)\leq \frac{\epsilon'}{4},
\end{align*}
for $1\leq q\leq k, 1\leq i\leq T_q, 1\leq j\leq s_q, 1\leq t\leq N_qC_{q,j}, 1\leq p\leq n(q,j)-1.$
Let $B(x(1,1,1,1),\cdots,x(k,s_k,T_k,N_kC_{k,s_k}))$
be the set of all shadowing points for the above pseudo-orbit. Precisely,
\begin{align*}
B(&x(1,1,1,1),\cdots,x(k,s_k,T_k,N_kC_{k,s_k}))=\\
B(&x(1,1,1,1),\cdots,x(1,1,1,N_1C_{1,1},),\cdots, x(1,s_1,1,1), \cdots, x(1,s_1,1,N_1C_{1,s_1});\\&\cdots;\\&
x(1,1,T_1,1),\cdots,x(1,1,T_1,N_1C_{1,1},),\cdots, x(1,s_1,T_1,1), \cdots, x(1,s_1,T_1,N_1C_{1,s_1});\\&\cdots;\\&
x(k,1,T_1,1),\cdots,x(k,1,1,N_kC_{k,1},),\cdots, x(k,s_k,1,1), \cdots, x(k,s_k,1,N_kC_{k,s_k});\\&\cdots;\\&
x(k,1,T_k,1),\cdots,x(k,1,T_k,N_kC_{k,1},),\cdots, x(k,s_k,T_k,1), \cdots, x(k,s_k,T_k,N_kC_{k,s_k})).
\end{align*}
Then the set $B(x(1,1,1,1),\cdots,x(k,s_k,T_k,N_kC_{k,s_k}))$ can be considered
as a map with variables $x(q,j,i,t)$.
We define $F_k$ by
\begin{align*}
F_k=\bigcup\{B(&x(1,1,1,1),\cdots,x(k,s_k,T_k,N_kC_{k,s_k})):\\
&x(1,1,1,1)\in W_{n(1,1)},\cdots,x(k,s_k,T_k,N_kC_{k,s_k})\in W_{n(k,s_k)}\}.
\end{align*}
Obviously, $F_k$ is non-empty compact and $F_{k+1}\subseteq F_{k}$. Define $F=\bigcap_{k=1}^{\infty}F_k$.

\begin{lem}\label{lem3.7}
For any $z\in F$,
\begin{align*}
\lim\limits_{k\to\infty}\mathscr{E}_{M_{2k}}(z)=\mu_1, \lim\limits_{k\to\infty}\mathscr{E}_{M_{2k+1}}(z)=\mu_2,
\end{align*}
where $M_q=\sum\limits_{r=1}^{q-1}(T_rY_r+s(r,1,r+1,1)),q=1,2,\cdots$.
\end{lem}
\begin{proof}
It siffices to prove that for any $\psi\in C^0(M)$,
\begin{align*}
\lim\limits_{k\to\infty} \left|\frac{1}{M_{k+1}}\sum\limits_{i=0}^{M_{k+1}-1}\psi(f^iz)-\int\psi d\mu_{\rho(k)}\right|=0.
\end{align*}
Assume that
\begin{align*}
z\in B(x(1,1,1,1),\cdots,x(k,s_k,T_k,N_kC_{k,s_k})).
\end{align*}
For $c>0$, let $\text{Var}(\psi,c)=\sup\{|\psi(x)-\psi(y)|:d(x,y)\leq c\}$.
We have
\begin{align*}
&\left|S_{T_kY_k+s(k,1,k+1,1)}\psi(f^{M_k}(z))-(T_kY_k+s(k,1,k+1,1))\int \psi d\mu_{\rho(k)}\right|\\
\leq&\left|S_{T_kY_k+s(k,1,k+1,1)}\psi(f^{M_k}(z))-T_kN_k\int\psi d\mu_{\rho(k)}\right|\\
&+\left|T_kN_k-(T_kY_k+s(k,1,k+1,1))\right|\|\psi\|\\
\leq&\left|S_{T_kY_k}\psi(f^{M_k}(z))-T_kN_k\int\psi d\nu_k\right|+T_kN_k\left|\int\psi d\nu_k-\int\psi d\mu_{\rho(k)}\right|\\
&+s(k,1,k+1,1)\|\psi\|+\left|T_kN_k-(T_kY_k+s(k,1,k+1,1))\right|\|\psi\|.
\end{align*}
Since $C_{k,j}n(k,j)=a_{k,j}$, we have
\begin{align*}
&\left|S_{T_kY_k}\psi(f^{M_k}(z))-T_kN_k\int\psi d\nu_k\right|\\
=&\left|S_{T_kY_k}\psi(f^{M_k}(z))-T_k\sum_{j=1}^{s_k}N_kC_{k,j}n(k,j)\int \psi dm_{k,j}\right|\\
\leq&\left|\sum_{i=1}^{T_k}\sum_{j=1}^{s_k}\sum_{t=1}^{N_kC_{k,j}} S_{n(k,j)}\psi(f^{M_{k,i,j,t}}(z))-
   T_k\sum_{j=1}^{s_k}N_kC_{k,j}n(k,j)\int \psi dm_{k,j}\right|+T_kX_{k}\|\psi\|\\
\leq&\left|\sum_{i=1}^{T_k}\sum_{j=1}^{s_k}\sum_{t=1}^{N_kC_{k,j}}\sum_{q=0}^{n(k,j)-1}
\psi(f^{M_{k,i,j,t}+q}(z))-\sum_{i=1}^{T_k}\sum_{j=1}^{s_k}\sum_{t=1}^{N_kC_{k,j}}\sum_{q=0}^{n(k,j)-1}
\psi(f^{q}(x(k,j,i,t)))\right|\\
&+\sum_{i=1}^{T_k}\sum_{j=1}^{s_k}\sum_{t=1}^{N_kC_{k,j}}
n(k,j)\left|\frac{1}{n(k,j)}S_{n(k,j)}\psi(x(k,j,i,t))-\int \psi dm_{k,j}\right|+T_kX_{k}\|\psi\|\\
\leq&\sum_{i=1}^{T_k}\sum_{j=1}^{s_k}\sum_{t=1}^{N_kC_{k,j}}n(k,j)\left|\frac{1}{n(k,j)}S_{n(k,j)}\psi(x(k,j,i,t))-\int \psi dm_{k,j}\right|\\
    &+T_kY_k\text{Var}(\psi,\frac{\epsilon'}{4}\exp(-\epsilon l_k))+T_kX_{k}\|\psi\|.
\end{align*}
By
\begin{align*}
D(\mu_{\rho(k)},\nu_k)\leq\frac{1}{k}, D(\mathscr{E}_{n(k,j)}(x(k,j,i,t)),m_{k,j})<\frac{1}{k},
\end{align*}
and inqualities (\ref{equ1}), (\ref{equ2}) and (\ref{equ3}), we have
\begin{align*}
\lim_{k\to\infty}\left|\frac{S_{T_kY_k+s(k,1,k+1,1)}\psi(f^{M_k}(z))}{T_kY_k+s(k,1,k+1,1)}-
\int \psi d\mu_k\right|=0.
\end{align*}
One can readily verify that
\begin{align*}
\lim_{k\to\infty}\frac{T_kY_k+s(k,1,k+1,1)}{M_{k+1}}=1.
\end{align*}
Since
\begin{align*}
&\left|\frac{1}{M_{k+1}}S_{M_{k+1}}\psi(z)-
\frac{S_{T_kY_k+s(k,1,k+1,1)}\psi(f^{M_k}(z))}{T_kY_k+s(k,1,k+1,1)}\right|\\
=&\left|\frac{1}{M_{k+1}}S_{M_k}\psi(z)
+\frac{S_{T_kY_k+s(k,1,k+1,1)}\psi(f^{M_k}(z))}{M_{k+1}}
-\frac{S_{T_kY_k+s(k,1,k+1,1)}\psi(f^{M_k}(z))}{T_kY_k+s(k,1,k+1,1)}\right|\\
=&\left|\frac{1}{M_{k+1}}S_{M_k}\psi(z)
+\frac{S_{T_kY_k+s(k,1,k+1,1)}\psi(f^{M_k}(z))}{T_kY_k+s(k,1,k+1,1)}
\left(\frac{T_kY_k+s(k,1,k+1,1)}{M_{k+1}}-1\right)\right|\\
\leq&\frac{M_k}{M_{k+1}}\|\psi\|+\|\psi\|\left|\frac{T_kY_k+s(k,1,k+1,1)}{M_{k+1}}-1\right|,
\end{align*}
we deduce that
\begin{align*}
\lim_{k\to\infty}\left|\frac{1}{M_{k+1}}S_{M_{k+1}}\psi(z)-\int \psi d\mu_k\right|=0,
\end{align*}
which completes the proof.
\end{proof}

From lemma \ref{lem3.7}, we have
$F\subset\widehat{M}(\varphi,f)\cap N(\widetilde{\Lambda}_\mu)$.
Next, we construct a sequence of measures to compute the topological entropy of $F$. We first undertake an intermediate
constructions. For each
\begin{align*}
\underline{x}=(x(1,1,1,1),\cdots,x(k,s_k,T_k,N_kC_{k,s_k}))\in W_{n(1,1)}\times\cdots\times W_{n(k,s_k)},
\end{align*}
we choose one point $z=z(\underline{x})$ such that
\begin{align*}
z\in B(x(1,1,1,1),\cdots,x(k,s_k,T_k,N_kC_{k,s_k}))
\end{align*}
Let $L_k$ be the set of all points constructed in this way. Fix the position indexed
$m,j,i,t,$ for distinct $x(m,j,i,t),x'(m,j,i,t)\in W_{n(m,j)},$ the corresponding shadowing points $z,z'$ satisfying
\begin{align*}
 &d(f^{M_{m,i,j,t}+q}(z), f^{M_{m,i,j,t}+q}(z'))\\
 \geq &d(f^q(x(m,j,i,t)),f^q(x'(m,j,i,t)))-d(f^{M_{m,i,j,t}+q}(z),f^q(x(m,j,i,t)))\\
       &-d(f^{M_{m,i,j,t}+q}(z'),f^q(x'(m,j,i,t)))\\
 \geq &d(f^q(x(m,j,i,t)),f^q(x'(m,j,i,t)))-\frac{\epsilon'}{2}.
\end{align*}
Noticing that $x(m,j,i,t),x'(m,j,i,t)$ are $(n(m,j),\epsilon')$-separated, we obtain
$f^{M_{m,i,j,t }}(z)$, $f^{M_{m,i,j,t }}(z')$ are $(n(m,j),\epsilon'/2)$-separated.
Thus
\begin{align*}
\sharp L_k=(\sharp W_{n(1,1)}^{N_1C_{1,1}}\sharp W_{n(1,2)}^{N_1C_{1,2}}\cdots\sharp W_{n(1,s_1)}^{N_1C_{1,s_1}})^{T_1}\cdots
(\sharp W_{n(k,1)}^{N_kC_{k,1}}\sharp W_{n(k,2)}^{N_kC_{k,2}}\cdots\sharp W_{n(k,s_k)}^{N_kC_{k,s_k}})^{T_k}.
\end{align*}
We now define, for each $k$, an atomic measure centred on $L_k$. Precisely, let
\begin{align*}
\alpha_k=\frac{\sum_{z\in L_k}\delta_z}{\sharp L_k}
\end{align*}
In order to prove the main results of this paper, we present some lemmas.
\begin{lem}
Suppose $\nu$ is a limit measure of the sequence of probability
measures $\alpha_k.$ Then $\nu(F)=1.$
\end{lem}
\begin{proof}
Suppose $\nu=\lim_{k\to\infty}\alpha_{l_k}$ for $l_k\to\infty$. For any fixed $l$ and all
$p\geq0$, $\alpha_{l+p}(F_l)=1$ since $F_{l+p}\subset F_l$. Thus, $\nu(F_l)\geq\limsup_{k\to\infty}\alpha_{l_k}(F_l)=1$.
It follows that $\nu(F)=\lim_{l\to\infty}\nu(F_l)=1$.
\end{proof}

Let $\EuScript{B}=B_n(x,\frac{\epsilon'}{8})$ be an arbitrary ball which intersets $F$. Let $k$ be an unique number satisfies
$M_{k+1}\leq n<M_{k+2}$. Let $i\in\{1,\cdots,T_{k+1}\}$ be the unique number so
\begin{align*}
M_{k+1,i}\leq n<M_{k+1,i+1}.
\end{align*}
Here we appoint $M_{k+1,T_{k+1}+1}=M_{k+2,1}$. We assume that $i\geq2$, the simpler case $i=1$ is similar.

\begin{lem}
For $p\geq 1$,
\begin{align*}
\alpha_{k+p}(B_{n}(x,\frac{\epsilon'}{8}))\leq (\sharp L_k(\sharp W_{n(k+1,1)}^{N_{k+1}C_{k+1,1}}\sharp W_{n(k+1,2)}^{N_{k+1}C_{k+1,2}}\cdots\sharp W_{n(k+1,s_{k+1})}^{N_{k+1}C_{k+1,s_{k+1}}})^{i-1})^{-1}.
\end{align*}
\end{lem}
\begin{proof}
Case $p=1$. Suppose $\alpha_{k+1}(B_{n}(x,\frac{\epsilon'}{8}))>0$,
then $L_{k+1}\cap B_{n}(x,\frac{\epsilon'}{8})
\neq \emptyset$.
Let $z=z(\underline{x},\underline{x}_{k+1}),z^{\prime}=z(\underline{y},\underline{y}_{k+1})
\in L_{k+1}\cap B_{n}(x,\frac{\epsilon'}{8})$, where
\begin{align*}
\underline{x}&=(x(1,1,1,1),\cdots,x(k,s_k,T_k,N_kC_{k,s_k})),\\
\underline{y}&=(y(1,1,1,1),\cdots,y(k,s_k,T_k,N_kC_{k,s_k})),
\end{align*}
and
\begin{align*}
\underline{x}_{k+1}=(&x(k+1,1,1,1),\cdots,x(k+1,s_{k+1},i-1,N_{k+1}C_{k+1,s_{k+1}}),\\
                    &\cdots,x(k+1,s_{k+1},T_k,N_{k+1}C_{k+1,s_{k+1}}))\\
\underline{y}_{k+1}=(&y(k+1,1,1,1),\cdots,y(k+1,s_{k+1},i-1,N_{k+1}C_{k+1,s_{k+1}})\\
                    &\cdots,y(k+1,s_{k+1},T_k,N_{k+1}C_{k+1,s_{k+1}})).
\end{align*}
Since $d_{n}(z,z')<\frac{\epsilon'}{4}$, we have $\underline{x}=\underline{y}$ and
$x(k+1,1,1,1)=y(k+1,1,1,1),\cdots,
x(k+1,s_{k+1},i-1,N_{k+1}C_{k+1,s_{k+1}})=y(k+1,s_{k+1},i-1,N_{k+1}C_{k+1,s_{k+1}})$.
Thus we have
\begin{align*}
\alpha_{k+1}(B_{n}(x,\frac{\epsilon'}{8}))\leq
&\frac{(\sharp W_{n(k+1,1)}^{N_{k+1}C_{k+1,1}}\sharp W_{n(k+1,2)}^{N_{k+1}C_{k+1,2}}\cdots\sharp W_{n(k+1,s_{k+1})}^{N_{k+1}C_{k+1,s_{k+1}}})^{T_{k+1}-(i-1)}}{\sharp L_{k+1}}\\
=&\left(\sharp L_k(\sharp W_{n(k+1,1)}^{N_{k+1}C_{k+1,1}}\sharp W_{n(k+1,2)}^{N_{k+1}C_{k+1,2}}\cdots\sharp W_{n(k+1,s_{k+1})}^{N_{k+1}C_{k+1,s_{k+1}}})^{i-1}\right)^{-1}.
\end{align*}
Case $p>1$ is similar.
\end{proof}
Since $a_{l,j}=n(l,j)C_{l,j}$,
\begin{align*}
\sharp W_{n(l,j)}\geq\exp\left(n(l,j)(1-\gamma)(h^{Kat}_{m_{l,j}}(f,\epsilon')-4\gamma)\right),
\end{align*}
by lemma \ref{lem3.6}, we have
\begin{align*}
\sharp L_k=&\left(\sharp W_{n(1,1)}^{N_1C_{1,1}}\sharp W_{n(1,2)}^{N_1C_{1,2}}\cdots\sharp W_{n(1,s_1)}^{N_1C_{1,s_1}}\right)^{T_1}\cdots
\left(\sharp W_{n(k,1)}^{N_kC_{k,1}}\sharp W_{n(k,2)}^{N_kC_{k,2}}\cdots\sharp W_{n(k,s_k)}^{N_kC_{k,s_k}}\right)^{T_k}\\
\geq&\exp\left(\sum_{l=1}^{k}\sum_{j=1}^{s_l}T_lN_lC_{l,j}n(l,j)(1-\gamma)(h^{Kat}_{m_{l,j}}(f,\epsilon')-4\gamma)\right)\\
\geq&\exp\left(\sum_{l=1}^{k}T_lN_l(1-\gamma)(h^{Kat}_{\mu_{\rho(l)}}(f,\epsilon')-4\gamma)\right)\\
\geq&\exp\left(\sum_{l=1}^{k}T_lN_l(1-\gamma)(\mathbf{C}-5\gamma)\right)
\end{align*}
and
\begin{align*}
&\left(\sharp W_{n(k+1,1)}^{N_{k+1}C_{k+1,1}}\sharp W_{n(k+1,2)}^{N_{k+1}C_{k+1,2}}\cdots\sharp W_{n(k+1,s_{k+1})}^{N_{k+1}C_{k+1,s_{k+1}}}\right)^{i-1}\\
\geq&\exp\left(\sum_{j=1}^{s_{k+1}}(i-1)N_{k+1}C_{k+1,j}n(k+1,j)(1-\gamma)(h^{Kat}_{m_{k+1,j}}(f,\epsilon')-4\gamma)\right)\\
\geq&\exp\left((i-1)N_{k+1}(1-\gamma)(\mathbf{C}-5\gamma)\right).
\end{align*}
Hence we obtain
\begin{align*}
&\sharp L_k\left(\sharp W_{n(k+1,1)}^{N_{k+1}C_{k+1,1}}\sharp W_{n(k+1,2)}^{N_{k+1}C_{k+1,2}}\cdots\sharp W_{n(k+1,s_{k+1})}^{N_{k+1}C_{k+1,s_{k+1}}}\right)^{i-1}\\
\geq&\exp\bigg\{\sum_{l=1}^{k}T_lN_l(1-\gamma)(\mathbf{C}-5\gamma)
+(i-1)N_{k+1}(1-\gamma)(\mathbf{C}-5\gamma)\bigg\}\\
=&\exp\left\{\left(\sum_{l=1}^{k}T_lN_l+(i-1)N_{k+1}\right)(\mathbf{C}-5\gamma)(1-\gamma)\right\}\\
=&\exp\left\{n\left((\mathbf{C}-5\gamma)(1-\gamma)
-\frac{n-\sum_{l=1}^{k}T_lN_l-(i-1)N_{k+1}}{n}(\mathbf{C}-5\gamma)(1-\gamma)\right)\right\}.
\end{align*}
From (\ref{equ2}) and $i\geq2$, we have
\begin{align*}
n-\sum_{l=1}^{k}T_lN_l-(i-1)N_{k+1}
&=n-\sum_{l=1}^{k}T_lY_l-(i-1)Y_{k+1}+\sum_{l=1}^{k}T_lX_l+(i-1)X_{k+1}\\
&\leq Y_{k+1}+\sum_{r=1}^{k+1}s(r,1,r+1,1)+\sum_{l=1}^{k}T_lX_l+(i-1)X_{k+1}.
\end{align*}
By inqualities (\ref{equ1}), (\ref{equ2}), (\ref{equ3}) and (\ref{equ4}) and $i\geq2$, we obtain
\begin{align*}
\lim_{n\to\infty}\frac{n-\sum_{l=1}^{k}T_lN_l-(i-1)N_{k+1}}{n}(\mathbf{C}-5\gamma)(1-\gamma)=0.
\end{align*}
Thus for sufficiently large $n$, we can deduce that
\begin{align*}
\limsup_{m\to\infty}\alpha_{m}(B_{n}(x,\frac{\epsilon'}{8}))\leq &(\sharp L_k(\sharp W_{n(k+1,1)}^{N_{k+1}C_{k+1,1}}\sharp W_{n(k+1,2)}^{N_{k+1}C_{k+1,2}}\cdots\sharp W_{n(k+1,s_{k+1})}^{N_{k+1}C_{k+1,s_{k+1}}})^{i-1})^{-1}\\
\leq&\exp\{-n((\mathbf{C}-5\gamma)(1-\gamma)-\gamma)\}.
\end{align*}
Applying the entropy distribution principle, we have
\begin{align*}
h_{top}(\widehat{M}(\varphi|N(\widetilde{\Lambda}_\mu),f),\frac{\epsilon'}{8})
\geq h_{top}(F,\frac{\epsilon'}{8})\geq (\mathbf{C}-5\gamma)(1-\gamma)-\gamma.
\end{align*}
Let $\epsilon'\to0$ and $\gamma\to0$; we have
\begin{align*}
h_{top}\left(\widehat{M}(\varphi|N(\widetilde{\Lambda}_\mu),f)\right)\geq\sup\left\{h_\nu(f):\nu\in \mathscr{M}_{inv}(\widetilde{\Lambda}_\mu,f)\right\}.
\end{align*} 

To obtain the upper bound, we need the following lemma.

\begin{lem}{\rm\cite{Bow}}
For $t\geq0$, consider the set
\begin{align*}
B(t)=\{x\in M:\exists \nu\in V(x) \text{ satisfying } h_\nu(f)\leq t\}.
\end{align*}
Then $h_{top}(B(t))\leq t$.
\end{lem}

Let
\begin{align*}
t=\sup\left\{h_\nu(f):\nu\in \mathscr{M}_{inv}(\widetilde{\Lambda}_\mu,f)\right\}.
\end{align*}
Then $\widehat{M}(\varphi|N(\widetilde{\Lambda}_\mu),f)\subset B(t)$.
Thus
\begin{align*}
h_{top}\left(\widehat{M}(\varphi|N(\widetilde{\Lambda}_\mu),f)\right)\leq\sup\left\{h_\nu(f):\nu\in \mathscr{M}_{inv}(\widetilde{\Lambda}_\mu,f)\right\},
\end{align*}
and the proof of theorem \ref{thm2.1} is completed.

\section{Some Applications}
{\bf Example 1 Diffeomorphisms on surfaces}
Let $f:M\to M$ be a $C^{1+\alpha}$ diffeomorphism with $\dim M=2$ and
$h_{top}(f)>0$, then there exists a hyperbolic measure $m\in\mathscr M_{\rm erg}(M,f)$
with Lyapunov exponents $\lambda_1>0>\lambda_2$(see \cite{Pol}). If
$\beta_1=|\lambda_2|$ and $\beta_2=\lambda_1$, then for any $\epsilon>0$
such that $\beta_1,\beta_2>\epsilon$, we have $m(\Lambda(\beta_1,\beta_2,\epsilon))=1$.
Let
\begin{align*}
\widetilde{\Lambda}=\bigcup_{k=1}^{\infty}\text{supp}(m|\Lambda(\beta_1,\beta_2,\epsilon)).
\end{align*}
If $\varphi\in C^0(M)$, one of the following conclusions is right.
\begin{enumerate}
\item The function $\mu\mapsto\int\varphi d\mu$ is constant for $\mu\in \mathscr M_{\rm inv}(\widetilde{\Lambda},f)$.
\item $\widehat{M}(\varphi|N(\widetilde{\Lambda}),f)\neq\emptyset$ and
      $h_{top}\left(\widehat{M}(\varphi|N(\widetilde{\Lambda}),f)\right)=\sup\left\{h_\mu(f):\mu\in\mathscr M_{\rm inv}(\widetilde{\Lambda},f)\right\}$.
\end{enumerate}

\noindent
{\bf Example 2 Nonuniformly hyperbolic systems} In \cite{Kat1}, Katok described a construction
of a diffeomorphism on the 2-torus $\mathbb{T}^2$ with nonzero Lyapunov exponents, which is not an Anosov map.
Let $f_0$ be a linear automorphism given by the matrix
\begin{align*}
  A=\begin{pmatrix}
   2&1\\
   1&1
  \end{pmatrix}
\end{align*}
with eigenvalues $\lambda^{-1}<1<\lambda$.
$f_0$ has a maximal measure $\mu_1$. Let $D_r$ denote the disk of radius $r$ centered
at (0,0), where $r>0$ is small, and put coordinates $(s_1,s_2)$ on $D_r$ corresponding to
the eigendirections of $A$, i.e, $A(s_1,s_2)=(\lambda s_1,\lambda^{-1}s_2)$.
The map $A$ is the time-1 map of the local flow in $D_r$
generated by the following system of differential equations:
\begin{align*}
\frac{ds_1}{dt}=s_1\log\lambda, \frac{ds_2}{dt}=-s_2\log\lambda.
\end{align*}
The Katok map is obtained from $A$ by slowing down these equations near the origin.
It depends upon a real-valued function $\psi$, which is defined on the unit interval $[0,1]$
and has the following properties:
\begin{enumerate}
\item[(1)] $\psi$ is $C^\infty$ except at 0;
\item[(2)] $\psi(0)=0$ and $\psi(u)=1$ for $u\geq r_0$ where $0<r_0<1$;
\item[(3)] $\psi^{\prime}(u)>0$ for every $0<u<r_0$;
\item[(4)] $\int_0^1\frac{du}{\psi(u)}<\infty$.
\end{enumerate}
Fix sufficiently small numbers $r_0<r_1$ and consider the time-1 map $g$ generated
by the following system of differential equations in $D_{r_1}$:
\begin{align*}
\frac{ds_1}{dt}=s_1\psi(s_1^2+s_2^2)\log\lambda, \frac{ds_2}{dt}=-s_2\psi(s_1^2+s_2^2)\log\lambda.
\end{align*}
The map $f$, given as $f(x)=g(x)$ if $x\in D_{r_1}$ and $f(x)=A(x)$ otherwise, defines a homeomorphism of
torus, which is a $C^\infty$ diffeomorphism everywhere except for the origin. To provide the differentiability
of map $f$, the function $\psi$ must satisfy some extra conditions. Namely, the integral $\int_0^1du/\psi$ must
converge ``very slowly" near the origin. We refer the smoothness to \cite{Kat1}. Here $f$ is contained in the
$C^0$ closure of Anosov diffeomorphisms and even more there is a homeomorphism $\pi:\mathbb{T}^2\to \mathbb{T}^2$
such that $\pi\circ f_0=f\circ \pi$. Let $\nu_0=\pi_\ast\mu_1$.

In \cite{LiaLiaSUnTia}, the authors proved
that there exist $0< \epsilon\ll \beta$ and a neighborhood $U$ of $\nu_0$ in
$\mathscr M_{\rm inv}(\mathbb{T}^2,f)$ such that for any ergodic $\nu\in U$ it holds that
$\nu\in \mathscr M_{\rm inv}(\widetilde{\Lambda}(\beta,\beta,\epsilon),f)$, where
$\widetilde{\Lambda}(\beta,\beta,\epsilon)=\bigcup_{k\geq1}{\rm supp}(\nu_0|\Lambda_k(\beta,\beta,\epsilon))$.

\begin{cro}
If $\varphi\in C^0(\mathbb{T}^2)$, one of the following conclusions is right.
\begin{enumerate}
\item The function $\mu\mapsto\int\varphi d\mu$ is constant for $\mu\in \mathscr M_{\rm inv}(\widetilde{\Lambda}(\beta,\beta,\epsilon),f)$.
\item $\widehat{\mathbb{T}^2}(\varphi|N(\widetilde{\Lambda}(\beta,\beta,\epsilon)),f)\neq\emptyset$ and
      $$h_{top}\left(\widehat{\mathbb{T}^2}(\varphi|N(\widetilde{\Lambda}(\beta,\beta,\epsilon)),f)\right)=
         \sup\left\{h_\mu(f):\mu\in\mathscr M_{\rm inv}(\widetilde{\Lambda}(\beta,\beta,\epsilon),f)\right\}.$$
\end{enumerate}
\end{cro}

In \cite{LiaLiaSUnTia}, the authors also studied the structure of Pesin set $\widetilde{\Lambda}$
for the robustly transitive partially hyperbolic diffeomorphisms described
by Ma\~{n}\'{e} and the robustly transitive non-partially hyperbolic
diffeomorphisms described by Bonatti-Viana.
They showed that for the diffeomorphisms derived from
Anosov systems $\mathscr M_{\rm inv}(\widetilde{\Lambda},f)$
enjoys many members. So our result is applicable to these maps.


\noindent {\bf Acknowledgements.}   The research was supported by
the National Basic Research Program of China
(Grant No. 2013CB834100) and the National Natural Science
Foundation of China (Grant No. 11271191).

\end{document}